\documentclass[reqno]{amsart}
\usepackage{etex}
\usepackage[a4paper,hmarginratio=1:1]{geometry}
\usepackage{amssymb,amsfonts,amsmath}
\usepackage{comment} 
\usepackage[all,arc]{xy}
\usepackage{enumerate}
\usepackage{mathrsfs,mathtools}
\usepackage{todonotes,booktabs}
\usepackage{stmaryrd}
\usepackage{marvosym}
\usepackage{graphicx}
\usepackage{pdflscape} 	
\usepackage{bbm}
\usepackage{tikz}
\usepackage{tikz-cd}
\usetikzlibrary{trees}
\usetikzlibrary[shapes]
\usetikzlibrary[arrows]
\usetikzlibrary{patterns}
\usetikzlibrary{fadings}
\usetikzlibrary{backgrounds}
\usetikzlibrary{decorations.pathreplacing}
\usetikzlibrary{decorations.pathmorphing}
\usetikzlibrary{positioning}
	
\def\biblio{\bibliography{duality}\bibliographystyle{alpha}}

\usepackage{xcolor} 
\usepackage{graphicx}
\SelectTips{cm}{10}

\usepackage[pagebackref]{hyperref}

\definecolor{dark-red}{rgb}{0.5,0.15,0.15}
\definecolor{dark-blue}{rgb}{0.15,0.15,0.6}
\definecolor{dark-green}{rgb}{0.15,0.6,0.15}
\hypersetup{
    colorlinks, linkcolor=dark-red,
    citecolor=dark-blue, urlcolor=dark-green
}

\renewcommand*{\backref}[1]{}
\renewcommand*{\backrefalt}[4]{%
  \ifcase #1 %
No citations.
  \or
(cit. on p. #2).%
  \else
(cit on pp. #2).%
  \fi%
}

\usepackage[nameinlink,capitalise,noabbrev]{cleveref}

\newtheorem{thm}{Theorem}[section]
\newtheorem{theorem}{Theorem}[section]
\newtheorem*{theorem*}{Theorem}

\newtheorem{proposition}[thm]{Proposition}

\newtheorem{lemma}[thm]{Lemma}

\theoremstyle{definition}

\theoremstyle{remark}

\makeatletter
\let\c@equation\c@thm
\makeatother
\numberwithin{equation}{section}

\DeclareMathOperator{\colim}{colim}

\newcommand{\Q}{\mathbb{Q}}

\DeclareMathOperator{\Tr}{Tr}

\DeclareMathOperator{\free}{free}

\DeclareMathOperator{\rank}{rank}

\newcommand{\Z}{\mathbb{Z}}

\Crefname{figure}{Figure}{Figures}
\Crefname{assu}{Assumption}{Assumptions}
\Crefname{lem}{Lemma}{Lemmas}
\Crefname{thm}{Theorem}{Theorems}
\Crefname{thma}{Theorem}{Theorems}
\Crefname{prop}{Proposition}{Propositions}

\let\lim\relax

\DeclareMathOperator{\lim}{lim}

\newcommand{\bG}{\mathbb{G}}

\newcommand{\F}{\mathbb{F}}

\newcommand{\cL}{\mathcal{L}}

\DeclareMathOperator{\Spf}{Spf}

\DeclareMathOperator{\Level}{Level}

\title{Transfer ideals and torsion in the Morava $E$-theory of abelian groups}

\author{Tobias Barthel}
\address{Max Planck Institute for Mathematics, Vivatsgasse 7, 53111 Bonn, Germany}
\email{tbarthel@mpim-bonn.mpg.de}
\thanks{The authors would like to thank the Max Planck Institute for Mathematics for its hospitality and Jeremy Hahn for helpful conversations. The second author was supported by NSF grant DMS-1906236.}
\author{Nathaniel Stapleton}
\address{University of Kentucky\\ Lexington, Kentucky, USA}
\email{nat.j.stapleton@uky.edu}

\subjclass[2010]{55N22 (primary); 55P42, 55S25 (secondary)}

\date{\today}
\begin{document}

\begin{abstract}
Let $A$ be a finite abelian $p$ group of rank at least $2$. We show that $E^0(BA)/I_{tr}$, the quotient of the Morava $E$-cohomology of $A$ by the ideal generated by the image of the transfers along all proper subgroups, contains $p$-torsion. The proof makes use of transchromatic character theory.
\end{abstract}

\maketitle

\setcounter{tocdepth}{1}
\def\biblio{}

\section{Introduction}

The close relationship between the Morava $E$-cohomology of spaces and algebraic geometry has played an important role in problems related to power operations for Morava $E$-theory, character theory for Morava $E$-theory, and understanding $H_\infty$-ring maps between $H_{\infty}$-ring spectra and $E$. Let $\bG$ be the universal deformation formal group associated to $E$ and write $I_{tr}$ for the transfer ideal. A fundamental result of Strickland's in \cite{etheorysym} shows that $E^0(B\Sigma_{p^k})/I_{tr}$ corepresents the formal scheme parametrizing subgroup schemes of order $p^k$ in $\bG$. For instance, this theorem plays a vital role in Rezk's study \cite{rezk_congruence, rezk_koszul} of the Dyer--Lashof algebra for Morava $E$-theory. One of the key ingredients in Strickland's proof is the fact that $E^0(B\Sigma_{p^k})/I_{tr}$ is a free $E^0$-module. 

Let $\Level(A^*, \bG)$ be the formal scheme of $A^*$-level structures in $\bG$, where $A^*$ is the Pontryagin dual of a finite abelian $p$-group $A$. In \cite{drinfeldell1}, it is shown that the ring of functions on $\Level(A^*, \bG)$ is a finitely generated free $E^0$-module. Proposition 7.5 in \cite{ahs} implies that there is a close relationship between $E^0(BA)/I_{tr}$ and the ring of functions on $\Level(A^*, \bG)$. In fact, if we let $(E^0(BA)/I_{tr})^{\free}$ be the image of $E^0(BA)/I_{tr}$ in $\Q \otimes E^0(BA)/I_{tr}$, then there is a canonical isomorphism
\[
\Spf((E^0(BA)/I_{tr})^{\free}) \cong \Level(A^*,\bG).
\]
For a transitive abelian subgroup $A \subseteq \Sigma_{p^k}$, Ando, Hopkins, and Strickland use this isomorphism to give an algebro-geometric description of the power operation
\[
E^0 \xrightarrow{P_{p^k}/I_{tr}} E^0(B\Sigma_{p^k})/I_{tr} \to E^0(BA)/I_{tr} \to (E^0(BA)/I_{tr})^{\free}.
\]

This motivates the question of whether $E^0(BA)/I_{tr}$ contains $p$-torsion, i.e., if 
\[
(E^0(BA)/I_{tr})^{\free} \ncong E^0(BA)/I_{tr}.
\]
These two rings are isomorphic when $A$ is cyclic. In contrast, the main result of this note implies that they are distinct when the rank of $A$ is greater than $1$.

\begin{theorem*}[\Cref{mainthm}]
Let $A$ be a finite abelian $p$-group with $\rank(A) \geq 2$. Let $I_{tr} \subset E^0(BA)$ be the ideal generated by the image of transfers along all proper subgroups of $A$. The $E^0$-algebra $E^0(BA)/I_{tr}$ contains $p$-torsion.
\end{theorem*}

In particular, this answers a question raised in \cite[Remark 5.3]{NatZhen}. Our proof uses the transchromatic character maps of \cite{Centralizers} to first reduce to a question about $p$-adic $K$-theory. We then use classical character theory to show that $K_{p}^0(BA)/I_{tr}$ is non-zero and contains torsion when $\rank(A) \geq 2$.

\section{$p$-adic $K$-theory of finite abelian $p$-groups and torsion}

Fix a prime $p$ and let $K_{p}$ denote $p$-adic $K$-theory. Given a finite $p$-group $G$, there is a canonical isomorphism
\[
\Z_p \otimes R(G) \xrightarrow{\cong} K_{p}^{0}(BG),
\]
where $R(G)$ is the representation ring of $G$. It follows that there is an isomorphism
\[
K_{p}^{0}(B\Z/p^k) \cong \Z_p[x]/(x^{p^k}-1) \cong \Z_p[x]/([p^k](x)), 
\]
where $[p^k](x)$ is the $p^k$-series for the multiplicative formal group law. Since $R(G \times H) \cong R(G) \otimes R(H)$ for finite groups $G$ and $H$, the isomorphisms above give an explicit description of $K_{p}^0(BA)$ for any finite abelian $p$-groups $A$.

Character theory for $p$-adic $K$-theory was developed by Adams in \cite[Section 2]{Adams-Maps2} and generalized to all Morava $E$-theories by Hopkins, Kuhn, and Ravenel in \cite{hkr}. Let $D = \colim_{i \geq 0} \Z_p(\zeta_{p^i})$, the ring obtained from $\Z_p$ by adjoining all of the $p^i$th roots of unity as $i$ varies. Let $C_0 = \Q \otimes D$. For a finite group $G$, the character map is a canonical inclusion of $\Z_p$-algebras
\[
K_{p}^{0}(BG) \hookrightarrow Cl_p(G,C_0),
\]
where $Cl_p(G,C_0)$ is the ring of $C_0$-valued functions on the set of conjugacy classes of $p$-power order elements in $G$. By construction, the character map factors through $Cl_p(G,D)$, the ring of $D$-valued functions on the set of conjugacy classes of $p$-power order elements in $G$, giving
\begin{equation} \label{factor}
K_{p}^{0}(BG) \hookrightarrow Cl_p(G,D) \hookrightarrow Cl_p(G,C_0).
\end{equation}
The ring $D$ is local with maximal ideal $(p) \subset D$.

Given $H \subseteq G$, there is a transfer map $K_{p}^{0}(BH) \to K_{p}^{0}(BG)$, which is a map of modules for the $K_{p}^{0}(BG)$-module structure coming from restriction. This transfer map is compatible with the transfer map on class functions
\[
\Tr_{H}^{G} \colon Cl_p(H,C_0) \to Cl_p(G,C_0)
\]
given by
\[
\Tr_{H}^{G}(f)([g]) = \frac{1}{|H|} \sum_{k \in G, kg k^{-1} \in H} f(k g k^{-1}),
\]
see for example \cite[Theorem D]{hkr}. Given a finite abelian $p$-group $A$ and a subgroup $A' \subseteq A$, the transfer map $K_{p}^0(BA') \to K_{p}^0(BA)$ is determined by the image of $1 \in K_{p}^0(BA')$ as the restriction $K_{p}^0(BA) \to K_{p}^0(BA')$ is surjective. Applying the formula above, we have
\begin{equation} \label{formula1}
\Tr_{A'}^{A}(1)(a) = 
\begin{cases}
|A/A'| & \text{if } a \in A' \\
0 & \text{otherwise}.
\end{cases}
\end{equation}
We define $I_{tr} \subset K_{p}^0(BA)$ to be the ideal generated by the image of the transfer maps along all proper subgroups of $A$. We will abuse notation and denote the ideal generated by transfers along all proper subgroups in $Cl_p(A,C_0)$ by $I_{tr}$ as well. In fact, since the class functions in \eqref{formula1} take values in the integers, we may also consider the ideal $I_{tr}$ in $Cl_p(A,D)$.

\begin{lemma} \label{prop1}
Assume $A$ is a finite abelian $p$-group, then the commutative ring $K_{p}^0(BA)/I_{tr}$ is non-zero.
\end{lemma}
\begin{proof}
This is clear when $A$ is cyclic. Now we assume that $\rank(A) \geq 2$. Since there is a map of commutative rings
\[
K_{p}^0(BA)/I_{tr} \rightarrow Cl_p(A,D)/I_{tr} \cong \left (\prod_{a \in A}D \right )/I_{tr},
\]
it suffices to prove that the target is non-zero. Consider the factor of $\prod_{a \in A}D$ corresponding to $0 \in A$. By the formula in \eqref{formula1}, the projection of $I_{tr}$ to that factor is contained in the ideal $(p) \subset D$.
\end{proof}

\begin{lemma} \label{prop2}
The commutative ring $K_{p}^0(B(\Z/p)^{\times 2})/I_{tr}$ is an $\F_p$-algebra.
\end{lemma}
\begin{proof}
By \cref{prop1}, it suffices to show that $(p) \subseteq I_{tr}$. Let $V_1, \ldots, V_{p+1}$ be the maximal abelian subgroups of $A=(\Z/p)^{\times 2}$, then 
\[
p = \left ( \sum_{i=1}^{p+1} \Tr_{V_i}^A(1) \right ) - \Tr_{0}^A(1).
\]
This can be checked in the ring of class functions. The value of the class function
\[
\sum_{i=1}^{p+1} \Tr_{V_i}^A(1)
\]
is $p$ on any non-zero element of $A$ and $p^2+p$ on $0 \in A$.
\end{proof}

\begin{proposition} \label{prop3}
Assume $A$ is a finite abelian $p$-group with $\rank(A) \geq 2$, then the commutative ring $K_{p}^0(BA)/I_{tr}$ is an $\F_p$-algebra.
\end{proposition}
\begin{proof}
Fix a surjection $\rho \colon A \twoheadrightarrow (\Z/p)^{\times 2}$. Since $\rho$ is surjective, given any proper subgroup $H \subset (\Z/p)^{\times 2}$, $\rho^{-1}(H)$ is a proper subgroup of $A$. Thus the map $\rho$ induces a map of commutative rings
\[
K_{p}^0(B(\Z/p)^{\times 2})/I_{tr} \to K_{p}^0(BA)/I_{tr}.
\]
It suffices to prove that $K_{p}^0(B(\Z/p)^{\times 2})/I_{tr}$ is an $\F_p$-algebra, which is the content of \cref{prop2}.
\end{proof}

\section{Morava $E$-theory of finite abelian $p$-groups and torsion}

Let $E$ be a height $n$ Morava $E$-theory at the prime $p$ and let $A$ be a finite abelian $p$-group. Just as before, we define the transfer ideal $I_{tr} \subset E^0(BA)$ to be the ideal generated by the image of the transfer maps along proper subgroups of $A$. 

Morava $E$-theory admits a variety of versions of character theory (e.g., \cite{hkr}, \cite{Greenlees-Strickland}, \cite{tgcm}, \cite{twisted}, \cite{Centralizers}). In \cite[Section 5]{Centralizers} the authors constructed a commutative ring $\bar{C}_1$, which depends on $n$ and $p$, that is a flat $E^0$-algebra and a faithfully flat $\Z_p$-algebra with the property that
\[
\bar{C}_1 \otimes_{E^0} E^0(BA) \cong \prod_{A^{\times n-1}} \bar{C}_1 \otimes_{\Z_p} K_{p}^0(BA).
\]
In fact, the isomorphism above can be extended to all finite groups but we will not need that generality here. Write $\cL$ for the free loop space functor. Since $\cL^{n-1}BA \simeq A^{\times n-1} \times BA$, we have an isomorphism
\[
\prod_{A^{\times n-1}} \bar{C}_1 \otimes_{\Z_p} K_{p}^0(BA) \cong \bar{C}_1 \otimes_{\Z_p} K_{p}^0(\cL^{n-1}BA).
\]

The isomorphism is compatible with transfers in the sense that there is a commutative diagram of abelian groups
\[
\xymatrix{E^0(BA') \ar[r] \ar[d]_{\Tr} & \bar{C}_1 \otimes_{\Z_p} K_{p}^0(\cL^{n-1}BA') \ar[d]^{\bar{C}_1 \otimes_{\Z_p} \Tr} \\ E^0(BA) \ar[r] &  \bar{C}_1 \otimes_{\Z_p} K_{p}^0(\cL^{n-1}BA),}
\]
where the transfer on the right is along the finite cover $\cL^{n-1}BA' \to \cL^{n-1}BA$. The commutativity of this diagram is due to the fact that it is induced by a map of cohomology theories. This is explained in Proposition 4.14 in \cite{BP}.

We say that a commutative ring $R$ contains torsion if the underlying $\Z$-module contains torsion. For a flat map $R \to S$ of commutative rings, an $R$-module $M$ contains $p$-torsion if $S \otimes_R M$ contains $p$-torsion. This follows from the fact that if $p\colon M \to M$ is injective then $p\colon S \otimes_R M \to S\otimes_RM$ is injective. 

\begin{theorem} \label{mainthm}
Assume that $A$ is an abelian $p$-group with $\rank(A) \geq 2$, then the commutative ring $E^0(BA)/I_{tr}$ contains $p$-torsion.
\end{theorem}
\begin{proof}
It follows from Theorem 6.9 in \cite{Centralizers} that
\[
\bar{C}_1 \otimes_{E^0} E^0(BA)/I_{tr} \cong \prod_{(a_i) \in A^{\times n-1}} \bar{C}_1 \otimes_{\Z_p} K_{p}^{0}(BA)/I_{tr}^{(a_i)},
\]
where $I_{tr}^{(a_i)}$ is a certain transfer ideal depending on the tuple $(a_i)$. The transfer ideals in the target were studied in \cite{NatZhen}. The ideal $I_{tr}^{(0)} \subset K_{p}^{0}(BA)$ is the ideal generated by transfers from all proper subgroups of $A$. \cref{prop3} implies that $K_{p}^{0}(BA)/I_{tr}^{(0)}$ has torsion. Since $\bar{C}_1$ is faithfully flat as a $\Z_p$-algebra, the ring
\[
\prod_{(a_i) \in A^{\times n-1}} \bar{C}_1 \otimes_{\Z_p} K_{p}^{0}(BA)/I_{tr}^{(a_i)}
\]
contains torsion, hence so does $\bar{C}_1 \otimes_{E^0} E^0(BA)/I_{tr}$. Since $\bar{C}_1$ is a flat $E^0$-algebra, it follows that $E^0(BA)/I_{tr}$ contains torsion as well.
\end{proof}

We conclude with a short comparison of \Cref{mainthm} with the existing literature. In \cite{etheorysym}, Strickland proves that $E^0(B\Sigma_{p^k})/I_{tr}$ is a free $E^0$-module, where $I_{tr}$ is the ideal generated by transfers along $\Sigma_i \times \Sigma_j \subset \Sigma_{p^k}$ with $i+j = p^k$ and $i,j>0$. Of course, $\Sigma_{p^k}$ is only abelian when $p=2$ and $k=1$, in which case it is cyclic. Further, given a transitive abelian subgroup $A \subseteq \Sigma_{p^k}$, there is an induced map of commutative rings
\[
E^0(B\Sigma_{p^k})/I_{tr} \to E^0(BA)/I_{tr}.
\]
However, due to the direction of the map, \cref{mainthm} does not have any consequences on the existence torsion in $E^0(B\Sigma_{p^k})/I_{tr}$. 

In \cite{genstrickland}, the main result of \cite{etheorysym} is generalized to groups of the form $A \wr \Sigma_{p^k}$. The transfer ideal $I_{tr} \subset E^0(BA \wr \Sigma_{p^k})$ considered there is generated by transfers along $A \wr (\Sigma_i \times \Sigma_j) \subseteq A \wr \Sigma_{p^k}$ with $i+j = p^k$ and $i,j>0$. They show that $E^0(BA \wr \Sigma_{p^k})/I_{tr}$ is a free $E^0$-module. In \cite{nelson}, Nelson proves that $E^0(B\Sigma_{p^k} \wr \Sigma_{p^l})/I_{tr}$ is free, where $I_{tr}$ is generated by transfers along the subgroups $(\Sigma_{i} \times \Sigma_j) \wr \Sigma_{p^l} \subset \Sigma_{p^k} \wr \Sigma_{p^l}$, where $i+j = p^k$ and $i,j>0$ and along the subgroups $\Sigma_{p^k} \wr (\Sigma_{i} \times \Sigma_j) \subset \Sigma_{p^k} \wr \Sigma_{p^l}$, where $i+j = p^l$ and $i,j>0$. In both of these cases, these results do not contradict our main result.

\biblio
\bibliography{bib}\bibliographystyle{alpha}

\end{document}